\documentclass[11pt]{amsart}
\usepackage[colorlinks=true,pagebackref,hyperindex,citecolor=green,linkcolor=blue]{hyperref}

\usepackage{amsmath}
\usepackage{amsfonts}
\usepackage{amssymb}
\usepackage{amsthm}
\usepackage{stmaryrd}
\usepackage[all]{xy}
\usepackage{mathrsfs}
\usepackage{enumitem} 
\usepackage[T1]{fontenc}
\usepackage{color}

\usepackage{amsmath,amssymb,amsthm}
\usepackage{setspace}
\usepackage{moreverb}
\usepackage{mathrsfs}
\usepackage{url}
\usepackage[all]{xy}
\usepackage{lscape}
\usepackage{tikz}
\usetikzlibrary{arrows}
\usepackage{color}

\newcommand{\ignore}[1]{}

\newtheorem{theorem}{Theorem}[section]

\newtheorem{corollary}[theorem]{Corollary}
\newtheorem{proposition}[theorem]{Proposition}

\theoremstyle{definition}
\newtheorem{definition}[theorem]{Definition}
\newtheorem{example}[theorem]{Example}
\newtheorem{conjecture}[theorem]{Conjecture}

\theoremstyle{remark}
\newtheorem{remark}[theorem]{Remark}

\newtheorem{discussion}[theorem]{Discussion}

\theoremstyle{notation}
\newtheorem{notation}[theorem]{Notation}

\numberwithin{equation}{section}

\input xy
\xyoption{all}

\newcommand{\bN}{{\mathbb{N}}}
\newcommand{\bZ}{{\mathbb{Z}}}

\definecolor{grey}{rgb}{0.75,0.75,0.75}
\definecolor{orange}{rgb}{1.0,0.5,0.5}
\definecolor{brown}{rgb}{0.5,0.25,0.0}
\definecolor{pink}{rgb}{1.0,0.5,0.5}

\newcommand{\hlt}{\mathrm {ht\, }}

\newcommand{\reg}{\mbox{\rm{reg}} \,}

\newcommand{\colim}{\mathrm{colim}}

\newcommand{\Hom}{\mbox{\rm{Hom}} }
\newcommand{\Ext}{\mbox{\rm{Ext}} }

\newcommand{\fM}{{\mathfrak m}}

\newcommand{\cP}{{\mathcal P}}
\newcommand{\cQ}{{\mathcal Q}}

\def\gin{{\mathrm{ gin}}}

\newcommand{\K}{\mathbb{K}}
\newcommand{\Z}{\mathbb{Z}}

\begin{document}

\title[Local cohomology of binomial edge ideals]{Local cohomology of binomial edge ideals and their generic initial ideals}

\author[J. \`Alvarez Montaner]{Josep \`Alvarez Montaner}

\address{Departament de Matem\`atiques\\
Universitat Polit\`ecnica de Catalunya\\ Av. Diagonal 647, Barcelona
08028, Spain} \email{Josep.Alvarez@upc.edu}

\thanks{Partially supported by Generalitat de Catalunya 2017 SGR-932 project  and
Spanish Ministerio de Econom\'ia y Competitividad
MTM2015-69135-P. the author is a member of the Barcelona Graduate
School of Mathematics (BGSMath).}



\begin{abstract}
We provide a Hochster type formula for the local cohomology modules of binomial edge ideals. As a consequence we obtain a simple criterion for the Cohen-Macaulayness and Buchsbaumness of these ideals and
we describe their Castelnuovo-Mumford regularity and their Hilbert series.  
Conca and Varbaro \cite{CV18}  have recently proved a conjecture of Conca, De Negri and Gorla \cite{CDG4} relating the graded components of the local cohomology modules of Cartwright-Sturmfels ideals and their generic initial ideals. We provide an alternative proof for the case of binomial edge ideals.
\end{abstract}

\maketitle

\section{Introduction}

Binomial edge ideals have been introduced  by  Herzog, Hibi, Hreinsdottir, Kahle, and Rauh in \cite{HHHKR} and independently by Ohtani in \cite{O}
as a way to associate a binomial ideal to a given simple graph  $G$ on the vertex set $[n]=\{1,\dots, n\}$. Binomial edge ideals can be also understood as 
a natural generalization of the ideal of $2$-minors of a $2\times n$-generic matrix. Over the last few years there has been a flurry of activity in the Combinatorial Commutative Algebra community
trying to translate algebraic properties of the ideals to combinatorial properties of the graph following the same spirit as in the research done in the case of monomial edge ideals.
Indeed, there has been a lot of effort in capturing the Cohen-Macaulay property of these ideals \cite{BMS,EHH, HHHKR, RR14, Rin13, Rin17}, 
to study the Betti numbers and the Castelnuovo-Mumford regularity \cite{dAH, EZ, HR18, JK, JNR, JNR2, MM, MR18-1,MR18-2, SK, ZZ}
or the Hilbert series \cite{BNB, KS, MS, SZ}. We should point out that most of the results mentioned here only work for specific classes of graphs such as {\it closed graphs}, {\it block graphs}, 
or {\it bipartite graphs} among others.

\vskip 2mm

An essential tool that has been used to address the same questions for monomial edge ideals is the celebrated Hochster's formula which originally appeared in \cite{Hoc} and one may consult different versions
in the books \cite{Sta96}, \cite{BH} and \cite{MS05}. Hochster's formula provide a decomposition of the local cohomology modules $H^r_{\fM}(A/I)$ of the Stanley-Reisner ring $A/I$ associated to a squarefree monomial ideal 
$I\subseteq A$, where  $A=\K[x_1,\dots, x_n]$  is a  polynomial ring with coefficients in a field $\K$. In an equivalent incarnation,  Hochster's formula  describes the Hilbert series of the local cohomology modules of $A/I$. The lack of such a formula for binomial edge ideals is one of the reasons that the questions mentioned above have been so challenging.

\vskip 2mm

In a different line of research,  Conca, De Negri and Gorla \cite{CDG1, CDG2, CDG3, CDG4} introduced a family of ideals inspired by work of Cartwright and Sturmfels in \cite{CS} which, roughly speaking,
are multigraded ideals  whose multigraded generic initial ideal is radical. They proved in \cite{CDG4} that binomial edge ideals are Cartwright-Sturmfels ideals computing explicitly the 
corresponding generic initial ideal. When comparing invariants of Cartwright-Sturmfels ideals with those of its generic initial ideals
they came up with the following:

\begin{conjecture}\label{conjH} \cite[Conjecture 1.14]{CDG4}
Let $I\subseteq A$ be a $\Z^m$-graded Cartwright-Sturmfels ideal and $\gin(I)$ its $\Z^m$-graded generic initial ideal, with $m\leq n$. Then one has: 
$$\dim_\K H^r_{\fM}(A/I)_a=\dim_\K H^r_{\fM}(A/\gin(I))_a$$
for every $r\in \bN$ and every $a\in \Z^m$. 
\end{conjecture}

Indeed, this conjecture implies the equality of the extremal Betti numbers of $I$ and $\gin(I)$, their 
projective dimension and  their Castelnuovo-Mumford regularity as well (see  also \cite[Conjecture 1.13]{CDG4}). A more general version of this conjecture has been proved recently by Conca and Varbaro in \cite{CV18}.

\vskip 2mm

The aim of this work is twofold,  on one hand we want to provide a Hochster type formula for the local cohomology modules of binomial edge ideals using the theory of local cohomology spectral sequences
developed by the author, Boix and Zarzuela  in \cite{ABZ}. On the other hand we want to prove Conjecture \ref{conjH} for binomial edge ideals using a completely different approach to the one considered in \cite{CV18}.
The advantage of our methods is that we provide a very explicit description of the graded pieces of these local cohomology modules. On the contrary, our techniques are very specific for this case of binomial 
edge ideals so we cannot expect to recover the general statement proved in \cite{CV18}.

\vskip 2mm

The structure of the paper is as follows. In Section \ref{spectral} we present the basics on local cohomology spectral sequences developed in \cite{ABZ}. Their description are technically involved so we will
try to avoid some of the details since, in the case that some extra assumptions on the ideal are satisfied, the spectral sequences become enormously simplified. This is the case that we will use throughout
this work. The version that we are going to consider is presented in Theorem \ref{sucesion espectral} and, as a consequence of the degeneration of this spectral sequence, we have a very general version of 
Hochster's formula (see Theorem \ref{Hochster}).

\vskip 2mm

In Section \ref{bei} we find the main results of this work.  After some work needed to verify that binomial edge ideals satisfy those extra assumptions, we obtain the main result of this work in Theorem \ref{main}
which is a Hochster type formula for the local cohomology modules of a binomial edge ideal. As an immediate consequence we obtain in Corollaries \ref{CM} and \ref{Buch} a simple criterion to decide the Cohen-Macaulayness and the Buchsbaumness
of binomial edge ideals. Moreover, we also obtain a formula for the $\bZ$-graded Hilbert series of these local cohomology modules in Theorem \ref{Hilbert}.

\vskip 2mm

In Section \ref{sec-gin} we turn our attention to the local cohomology modules of the generic initial ideal of a binomial edge ideal. Our main result is Theorem \ref{compare} where we settle Conjecture \ref{conjH}.
Our approach is to use a coarser version of the original Hochster formula for this very particular monomial ideal (see Theorem \ref{main2}) that has the same structure as the one obtained in Theorem \ref{main}.
In this way we reduce the problem to a simple comparision of the graded pieces of the building blocks of these decompositions. Finally we present a formula for the Castelnuovo-Mumford regularity of
binomial edge ideals in Theorem \ref{regularity} and a formula for the $\bZ^n$-graded Hilbert series  in Theorem \ref{Hilbert2}.

\vskip 2mm

{\it Aknowledgements:}  We greatly appreciate Aldo Conca, Emanuela de Negri and Elisa Gorla for bringing to our attention their work on Cartwright-Sturmfels ideals
and, in particular, Conjecture \ref{conjH} during  a workshop in honor of Peter Schenzel held in Osnabr\"uck. We also want to thank Alberto F. Boix and Santiago Zarzuela for so many discussions about 
local cohomology spectral sequences we had over the last few years. We also thank the anonymous referee for the helpful comments on the manuscript.

  \section{Local cohomology spectral sequence} \label{spectral}

Let $A$ be a commutative Noetherian ring containing a field $\K$ and let $(\cP,\preccurlyeq )$ be a partially ordered set of ideals $I_p\subseteq A$ ordered by reverse inclusion. Namely, given $p,q\in \cP$ we have
$p \preccurlyeq q$ if and only if $I_p \supseteq I_q$. We will assume that there exists a maximal ideal $\fM$ of $A$ containing all the ideals in the poset. In the sequel, if no confusion arise, we will indistinctly  refer to the ideal $I_p\in \cP$ or the  element $p\in \cP$ for simplicity. 

\vskip 2mm

A formalism to produce local cohomology spectral sequences associated to a poset was given in \cite{ABZ}. The construction of these spectral sequences is technically involved so we are just going to sketch the main ideas behind its construction and refer to \cite{ABZ} for details.
Associated to the poset $\cP$ we consider the inverse system  of $A$-modules $A/[*]:=(A/I_p)_{p\in \cP}$.  Applying the $\fM$-torsion functor $\Gamma_\fM$ to this inverse system we obtain another inverse system  $\Gamma_\fM(A/[*]):=(\Gamma_\fM(A/I_p))_{p\in \cP}$.
The category of inverse systems of $A$-modules over a poset has enough injectives so we may consider the right derived inverse systems
$\mathbb{R}^j \Gamma_\fM(A/[*])$.  The spectral sequence given in \cite [Theorem 5.6]{ABZ} involves the right derived functors of 
the limit of this inverse system. In the form we are interested in this work it reads as:
$$E_2^{i,j}= \mathbb{R}^i \lim_{p\in \cP}\xymatrix{\mathbb{R}^j \Gamma_\fM(A/[*])\ar@{=>}[r]& }H_\fM^{i+j}\left(\lim_{p\in
\cP}A/I_p\right).$$
Notice that the spectral sequence starts with an inverse system of $A$-modules but converges to the usual local cohomology module which lives in the category of $A$-modules.

\vskip 2mm

The poset that we will consider throughout this work is what we call {\it the poset $\cP_I$ associated to a decomposition of the ideal} $I$.
Namely, given  an ideal $I\subseteq A$ admitting a decomposition $I=I_1\cap\ldots\cap I_n$ for certain ideals $I_1,\ldots ,I_n$ of $A$,  the poset $\cP_I$  is given by all the possible sums of the ideals in the decomposition. So, any ideal $I_p\in \cP_I$ is just a certain sum
of  ideals among $I_1,\ldots ,I_n$.  We point out that we have a lot of flexibility on the decomposition that we may consider and we will make a strong use of this fact. Most commonly we will use a minimal primary decomposition of $I$ which is the approach considered in \cite{ABZ} (see also \cite{AGZ}) when dealing with monomial ideals or arrangements of linear subspaces. For the case of binomial edge ideals  we are going to consider a poset $\cQ_I$ which is a refinement of $\cP_I$ that adjusts much better to our purposes (see Definition \ref{posetQ}). We should mention that the inverse systems  $(A/I_p)_{p\in \cP_I}$ and 
$(A/I_q)_{q\in \cQ_I}$ are cofinal so there will be no problem when computing their limit (see Remark \ref{cofinal}).

\vskip 2mm

At first sight the spectral sequence seems quite complicated but it becomes enormously simplified in the case that we consider some extra assumptions on the poset $\cP$, being either $\cP_I$ or  $\cQ_I$, associated to a decomposition of the ideal $I$. 

\vskip 2mm

{\bf Key assumptions:} 

\begin{itemize}
\item[(A1)]  $\cP$  is a subset of a distributive lattice of ideals of $A$.

\vskip 2mm

\item[(A2)]  $A/I_p$ is Cohen-Macaulay for all $p\in \cP$.

\vskip 2mm

\item[(A3)] $I_q$ is not contained in any minimal prime of $I_p$ for any $p\neq q$ with $\hlt (I_q) > \hlt (I_p)$. 

\end{itemize}

\vskip 2mm
Recall that ideals in $A$ form a lattice with the intersection and the sum. We say that it is a distributive lattice when both operations distribute over each other. For example, given a set of ideals $I_1,\ldots ,I_n$ of $A$, we may consider the lattice formed by all the sums and all the intersections of these ideals. In particular, the poset $\cP$ that we consider may be interpreted as a subset of such a lattice.
A lattice of monomial ideals is always distributive but we may find examples where this is not the case (see for instance \cite[Remark 5.2]{ABZ}). 

\vskip 2mm

Assumption  (A1) implies that $\lim_{p\in \cP} A/I_p =A/I$ (see \cite[Proposition 5.1]{ABZ}). On the other hand, Assumptions  (A2) and  (A3) ensure that the spectral sequence degenerates at the  $E_2$-page
(see \cite[Theorem 5.22]{ABZ}).  Actually, (A2) and (A3) are satisfied whenever $A/I_p$ are Cohen-Macaulay domains for all $p\in \cP$.  

\vskip 2mm

\begin{remark} \label{rmkA3}
The condition of Assumption (A3) as it appears in \cite[Theorem 5.22]{ABZ} only requires that $p\neq q$. However, if we take a close look at the proof of \cite[Lemma 5.21]{ABZ}  we can see that it is enough checking the condition when $\hlt (I_q) > \hlt (I_p)$ in the case that all the minimal primes of $I_p$ have the same height, which occurs when $A/I_p$ is Cohen-Macaulay. 
%
%
\end{remark}

\vskip 2mm

In order to check  Assumption (A2)  it will be useful to consider the following result, due to  Bouchiba and Kabbaj \cite[Theorem 2.1]{BK}, 
on the Cohen-Macaulayness of the tensor product (see also \cite{EGA} and \cite{Wat}).

  \begin{theorem} \cite[Theorem 2.1]{BK} \label{CMtensor}
  Let $A$ and $B$ be $\K$-algebras such that $A \otimes_\K B$ is Noetherian. Then, the following conditions
are equivalent:
   
   \begin{itemize}
    \item[i)] $A$ and  $B$ are Cohen-Macaulay.
    
    \item[ii)] $A \otimes_\K B$ is Cohen-Macaulay.
   \end{itemize}

  \end{theorem}

%
%
%

\vskip 2mm

Let $1_{\cP}$ be a terminal element that we add to the poset $\cP$. To any $p\in \cP$ we may consider the {\it order complex} associated to the  subposet
$(p, 1_{{\cP}}):=\{z\in \cP \mid\quad p < z < 1_{\cP}\}$ and its reduced simplicial cohomology  groups 
$\widetilde{H}^{i} ((p,1_{{\cP}}); \K)$. 
Under Assuptions {\rm (A1)}, {\rm (A2)} and {\rm (A3)},  the spectral sequence for local cohomology modules reads as follows:

\newpage

\begin{theorem}\cite [Theorem 5.22]{ABZ}\label{sucesion espectral}
Let $A$ be a commutative Noetherian ring containing a field $\K$ and let $\cP$ be a poset 
associated to a decomposition of an ideal $I\subseteq A$ satisfying Assuptions {\rm (A1)}, {\rm (A2)} and {\rm (A3)}.  Then, there exists a first quadrant spectral sequence of the form:
\[
E_2^{i,j}=\bigoplus_{j=d_p} H_{\mathfrak{m}}^{d_p}
\left(A/I_p\right)^{\oplus M_{i,p}}\xymatrix{ \ar@{=>}[r]_-{i}& }
H_{\mathfrak{m}}^{i+j} \left(A/I\right),
\]
where $M_{i,p}=\dim_{\K} \widetilde{H}^{i-d_p-1} ((p,1_{{\cP}});
\K)$ and $d_p=\dim A/I_p$. Moreover, this spectral sequence degenerates at the $E_2$-page.
\end{theorem}

\vskip 2mm

From the degeneration of the spectral sequence we get that,  for each $0\leq r\leq\dim (A)$,
there is an increasing, finite filtration $\{H_i^r\}_{0\leq i \leq b}$  of
$H_{\mathfrak{m}}^r (A/I)$ by $A$-modules, or equivalently 
we have a collection of short exact sequences 
\begin{align*} \label{short_seq}
& \xymatrix{0\ar[r]& H_0^r\ar[r]& H_1^r\ar[r]& H_1^r/H_0^r\ar[r]& 0}\\
& \xymatrix{0\ar[r]& H_1^r\ar[r]& H_2^r\ar[r]& H_2^r/H_1^r\ar[r]& 0}\\
& \xymatrix{& {\hskip 4mm\vdots}& &{\hskip -8mm \vdots}& &{\hskip -15mm \vdots}& & }\\
&  \xymatrix{0\ar[r]& H_{b-1}^r\ar[r]& H_b^r\ar[r]& H_b^r/H_{b-1}^r\ar[r]& 0.}
\end{align*}

\noindent where $H_b^r=H_{\mathfrak{m}}^r (A/I)$ for some $b\in \bN$ and, for each $r$ the quotients $H_i^r/H_{i-1}^r$ 
decompose in the following manner:
\[
H_i^r/H_{i-1}^r\cong\bigoplus_{\{p\in \cP\ \mid\ r-i=d_p\}}
H_{\mathfrak{m}}^{d_p} \left(A/I_p\right)^{\oplus M_{i,p}}
\]

\vskip 2mm

These short exact sequences split as $\K$-vector spaces so we obtain
the following generalization of the celebrated Hochster's formula.

\begin{theorem}\label{Hochster}
Let $A$ be a commutative Noetherian ring containing a field $\K$ and let $\cP$ be a poset 
associated to an ideal $I\subseteq A$ satisfying Assuptions (A1), (A2) and (A3).  Then, 
there is a $\K$-vector space isomorphism
\[
H_{\mathfrak{m}}^{r} \left(A/I\right)\cong\bigoplus_{p\in \cP} H_{\mathfrak{m}}^{d_p} \left(A/I_p\right)^{\oplus M_{r,p}}
\]

\end{theorem}

It is proved in \cite[Sections 5.3 and 7.1]{ABZ}  that in the polynomial ring case we may add an enhanced structure to the spectral sequence  
considered in Theorem \ref{sucesion espectral}. In our case we want to consider a $\Z^m$-graded structure where $m\leq n$.
Then, the graded version of Theorem \ref{Hochster} reads as 
\[
H_{\mathfrak{m}}^{r} \left(A/I\right)_a\cong\bigoplus_{p\in \cP} H_{\mathfrak{m}}^{d_p} \left(A/I_p\right)_a^{\oplus M_{r,p}}
\]
for every $a\in \Z^m$.

\vskip 2mm

The original Hochster formula is nothing but a very particular case of Theorem \ref{Hochster}.

\begin{corollary}\label{Hochster_mono}
Let $A=\K[x_1,\dots, x_n]$ be a polynomial ring with coefficients over a field $\K$ and let $I\subseteq A$
be a squarefree monomial ideal. Let $\cP_I$ be the poset given by all the possible sums of the ideals in the minimal primary decomposition 
of $I$.  Then, 
there is a $\K$-vector space isomorphism
\[
H_{\mathfrak{m}}^{r} \left(A/I\right)\cong\bigoplus_{p\in \cP} H_{\mathfrak{m}}^{d_p} \left(A/I_p\right)^{\oplus M_{r,p}}.
\] Moreover we have a decomposition as graded  $\K$-vector spaces.
\end{corollary}

\section{Local cohomology modules of binomial edge ideals} \label{bei}

Let $A=\K[x_1,\dots, x_n, y_1,\dots, y_n]$ be a polynomial ring with coefficients over a field $\K$ and let $G$ be a graph on the vertex set $[n]=\{1,\dots, n\}$. The {\it binomial edge ideal} associated to $G$ is the ideal 
$$J_G= \langle \Delta_{ij}  \hskip 2mm | \hskip 2mm \{i,j\} \mbox{ is an edge of } G \rangle,$$ 
where $\Delta_{ij}=   x_iy_j-x_jy_i$ denote the corresponding $2$-minor of a  $2\times n$ matrix 

$$\left(
\begin{array}{cccc}
x_1 & x_2 & \cdots & x_n \\
y_1 & y_2 & \cdots & y_n 
\end{array}
\right)
$$


With the $\Z^n$-graded structure on $A$ given  by $\deg(x_i)=\deg(y_i)=e_i\in \Z^n$, where $e_i$ is the $i$-th unit vector, we have that $J_G$ is a $\Z^n$-graded ideal.
We may also associate the usual $\Z$-graded structure given  by $\deg(x_i)=\deg(y_i)=1$.

\vskip 2mm


Binomial edge ideals are radical ideals whose primary decomposition  is nicely described. Namely, given $S\subseteq [n]$, let $c=c(S)$ be the number of connected components of $G\setminus S$, that we denote by $G_1,\dots, G_c$, and  let $\widetilde{G}_i$ be the complete graph on the vertices of $G_i.$ We set
$$P_S(G)= \langle x_i,y_i \hskip 2mm | \hskip 2mm i\in S \rangle + J_{\widetilde{G}_1} +\cdots + J_{\widetilde{G}_{c}}$$
and  let $|S|$ be the number  of elements of $S$.

\begin{theorem}[{\cite{HHHKR}}]\label{IntersectionPrimes}
Let $G$ be a graph on  $[n]$, and $J_G$ its binomial edge ideal in $A$. Then we have:

\vskip 2mm

\begin{itemize}
\item[i)] $P_S(G)$ is a prime ideal  of height $n-c+|S|$ for every $S\subseteq [n]$.
\vskip 2mm

\item[ii)] We have a decomposition $$J_G=\bigcap_{S\subseteq [n]} P_S(G).$$

%

\vskip 2mm

\item[iii)]  If $G$ is a connected graph on $[n]$, then $P_S(G)$ is a minimal prime of $J_G$ if and only if either $S=\emptyset$ or $S\neq \emptyset$ and, for each $i\in S$, 
we have $c(S\setminus \{i\}) < c(S)$.

\end{itemize}

\end{theorem}

 \subsection{A poset associated to a binomial edge ideal} Let $J_G$ be the binomial edge ideal associated to a graph $G$  on the set of vertices $[n]$.
 Associated to the minimal primary decomposition 
$$J_G=P_{S_1}(G) \cap \cdots \cap P_{S_r}(G)$$ 
we may consider the poset $\cP_{J_G}$ given by all the possible sums of the ideals in the decomposition ordered by reverse inclusion. 
An interesting feature that occurs in the case of monomial ideals or arrangement of linear subspaces is that the corresponding poset only contains prime ideals.
This is no longer true for the case of binomial edge ideals as the following example shows.
%
%

\vskip 2mm

\begin{example} \label{path}
Let $\K[x_1,\dots, x_5, y_1,\dots, y_5]$ be a polynomial ring with coefficients over a field $\K$ and 
let $J_G$ be the binomial edge ideal associated to the $5$-path  $G:=\{ \{1,2\}, \{2,3\},\{3,4\},\{4,5\} \} $. 
Its minimal primary decomposition is $J_G=I_1\cap I_2\cap I_3\cap I_4 \cap I_5$ where:

\vskip 2mm

$I_1:= P_{\emptyset}(G)= \langle  \Delta_{12},  \Delta_{13}, \Delta_{14}, \Delta_{15},  \Delta_{23}, \Delta_{24}, \Delta_{25},  \Delta_{34}, \Delta_{35},  \Delta_{45} \rangle$

$I_2:= P_{\{2\}}(G)= \langle x_2,y_2, \Delta_{34}, \Delta_{35},  \Delta_{45}  \rangle$

$I_3:= P_{\{3\}}(G)= \langle x_3,y_3, \Delta_{12},   \Delta_{45}  \rangle$

$I_4:= P_{\{4\}}(G)= \langle x_4,y_4, \Delta_{12},  \Delta_{13}, \Delta_{23}  \rangle$

$I_5:= P_{\{2,4\}}(G)= \langle x_2,y_2, x_4, y_4  \rangle$

\vskip 2mm

\noindent We have that $I_2+I_4+I_5= \langle x_2,y_2, x_4, y_4,  \Delta_{13}, \Delta_{35} \rangle$  admits a primary decomposition of the form
$$I_2+I_4+I_5= \langle x_2,y_2, x_4, y_4,  \Delta_{13}, \Delta_{15}, \Delta_{35} \rangle \cap \langle x_2,y_2, x_3,y_3, x_4, y_4  \rangle. $$
We point out that these prime components as well as its sum $\langle x_2,y_2, x_3,y_3,x_4, y_4,  \Delta_{15}\rangle$ are indeed elements of the poset $\cP_{J_G}$. 
The complete list of ideals in the poset $\cP_{J_G}$ contains $I_1, I_2, I_3, I_4, I_5$ and  $\langle x_2,y_2,  \Delta_{13}, \Delta_{14}, \Delta_{15}, \Delta_{34}, \Delta_{35},  \Delta_{45}  \rangle$, 
$\langle x_3,y_3,  \Delta_{12}, \Delta_{14}, \Delta_{15}, \Delta_{24}, \Delta_{25},  \Delta_{45}  \rangle$,
$\langle x_2,y_2, x_3, y_3, \Delta_{45}  \rangle$,
$\langle x_2,y_2, x_4, y_4, \Delta_{13}  \rangle$,
$\langle x_4,y_4, \Delta_{12},  \Delta_{13}, \Delta_{15}, \Delta_{23},  \Delta_{25},  \Delta_{35} \rangle$,
$\langle x_2,y_2, x_4, y_4, \Delta_{35}  \rangle$,
$\langle x_3,y_3, x_4, y_4, \Delta_{12}  \rangle$,
{\color{blue} $\langle x_2,y_2, x_4, y_4,  \Delta_{13}, \Delta_{35} \rangle$},
$\langle x_2,y_2, x_3,y_3, x_4, y_4  \rangle$,
$\langle x_2,y_2, x_3, y_3,  \Delta_{14}, \Delta_{15}, \Delta_{45} \rangle$,
$\langle x_2,y_2, x_4, y_4,  \Delta_{13}, \Delta_{15}, \Delta_{35} \rangle$,
$\langle x_3,y_3, x_4, y_4,  \Delta_{12}, \Delta_{15}, \Delta_{25} \rangle$,
$\langle x_2,y_2, x_3,y_3, x_4, y_4, \Delta_{15}  \rangle$.

\vskip 2mm

All the ideals in the poset, except for the one presented in blue, are prime ideals of the form $P_S(H)$ for some graph $H$ in the vertex set $[n]$
and some $S\subseteq [n]$.

\end{example}

The fact that some elements of the poset $\cP_{J_G}$ may not be prime ideals is not a big issue in the computation of the corresponding local cohomology modules. Indeed, the building blocks $ H_{\mathfrak{m}}^{d_p} \left(A/I_p\right)$  associated to non prime ideals appearing in the  formula given in Theorem \ref{Hochster} admit some further decomposition applying Theorem \ref{Hochster} as many times as necessary.
%
%
%
 Is for this reason that it seems more natural to associate the following poset to the minimal primary decomposition of the ideal ${J_G}$ containing only prime ideals and being cofinal with $\cP_{J_G}$. 
 
\begin{definition} \label{posetQ}
Let $J_G$ be the binomial edge ideal associated to a graph $G$ on the set of vertices $[n]$. We construct iteratively a poset  $\cQ_{J_G}$ associated to $J_G$ as follows:  The ideals contained  in $\cQ_{J_G}$ are the prime ideals in $\cP_{J_G}$,   the prime ideals in the posets $\cP_{I_q}$ of sums of the ideals in the decomposition of every non prime ideal $I_q$ in $\cP_{J_G}$ and the prime ideals that we obtain repeating this procedure every time we find a non prime ideal. 
%
%
This process stops after a finite number of steps since we only have a finite number of prime ideals of the form $P_S(H)$ for some graph $H$ on the  set of vertices  $[n]$ and some $S\subseteq [n]$.

\end{definition}

\begin{remark}
In Example \ref{path},  we have that the poset $\cQ_{J_G}$ contains all the ideals in $\cP_{J_G}$ except for {\color{blue} $\langle x_2,y_2, x_4, y_4,  \Delta_{13}, \Delta_{35} \rangle$}.
\end{remark}

\vskip 2mm

Before going on, we will fix some notation that will be useful in the rest of this work.

\vskip 2mm

\begin{notation} \label{notation}
 Let $I_q$ be a prime ideal in the poset $\cQ_{J_G}$. Assume that it has the form
$$I_q=P_S(H)= \langle x_i,y_i \hskip 2mm | \hskip 2mm i\in S \rangle + J_{\widetilde{H}_1} +\cdots + J_{\widetilde{H}_{c_q}},$$
where $H$ is a graph in the set of vertices $[n]$, $S\subseteq [n]$,  $H_1,\dots, H_{c_q}$ are the connected components of $H\setminus S$ and $\widetilde{H}_i$ denotes the complete graph on the vertices of $H_i$.  
Let $n_{i,q}$ be the number of vertices of $H_i$, then we may rename and reorder the variables $x_i, y_i$  in such a way that

\vskip 2mm 

$\langle x_i,y_i \hskip 2mm | \hskip 2mm i\in S \rangle\subseteq A_0:=\K[x_j,y_j \hskip 2mm | \hskip 2mm j\in S]=\K[x_1^0,\dots , x_{|S|}^0,y_1^0,\dots , y_{|S|}^0]$

\vskip 2mm 

\noindent and, for $i=1,\dots,c_q$

\vskip 2mm 

$J_{\widetilde{H}_i} \subseteq A_i:= \K[x_j, y_j \hskip 2mm | \hskip 2mm j \in H_i]=\K[x_1^i,\dots , x_{n_{i,q}}^i,y_1^i,\dots , y_{n_{i,q}}^i]$.

\vskip 2mm 

\noindent In particular, we have $$A=A_0 \otimes_\K A_1  \otimes_\K \cdots \otimes_\K A_{c_q}$$   
$$A/I_q =  A_1/J_{\widetilde{H}_1}  \otimes_\K \cdots \otimes_\K A_{c_q}/J_{\widetilde{H}_{c_q}} .$$
Notice that $d_q:= \dim A/I_q = n - |S| +c =d_1+ \cdots + d_{c_q}$, where  $d_i:= \dim A_i/J_{\widetilde{H}_i }= n_{i,q} + 1$. 
In the sequel we will  denote by $\fM_i$ the homogeneous maximal ideal of $A_i$.

\end{notation}


\subsection{Local cohomology of binomial edge ideals}
In order to  apply Theorem \ref{Hochster} for binomial edge ideals we have to make sure that the  poset $\cQ_{J_G}$ that we constructed in Definition \ref{posetQ} satisfies Assumptions (A1), (A2) and (A3).

\vskip 2mm

Recall that the ideal of $2\times 2$-minors of a $2\times n$-matrix is Cohen-Macaulay and thus  $A/I_q $ is also Cohen-Macaulay by Theorem \ref{CMtensor}. Therefore Assumptions (A2) and (A3) are satisfied for the poset $\cQ_{J_G}$ since it only contains Cohen-Macaulay prime ideals. 
In order to prove Assumption (A1) we will first do the following

\begin{discussion} \label{discussion}
 Let $I_{q_1}$ and $I_{q_2}$ be a prime ideals in the poset $\cQ_{J_G}$ of the form
$$I_{q_1}=P_{S_1}(H_1)= \langle x_i,y_i \hskip 2mm | \hskip 2mm i\in S_1 \rangle + J_{\widetilde{H}_{1,1}} +\cdots + J_{\widetilde{H}_{1,c_{q_1}}},$$
$$I_{q_2}=P_{S_2}(H_2)= \langle x_j,y_j \hskip 2mm | \hskip 2mm j\in S_2 \rangle + J_{\widetilde{H}_{2,1}} +\cdots + J_{\widetilde{H}_{2,c_{q_2}}},$$
for some graphs $H_1$ and $H_2$ in the set of vertices $[n]$. In order to study the intersection $I_{q_1}\cap I_{q_2}$ we notice that both ideals are sums of ideals in different sets of separated variables and thus we may reduce the problem to the study of the intersection of either two 
prime monomial ideals, a prime monomial ideal and a binomial edge ideal of a complete graph or the intersection of two binomial edge ideal of  complete graphs. In these cases we have the following description of the generators in terms of the corresponding graphs:

$\cdot$ A minimal set of generators of $  \langle x_i,y_i \hskip 2mm | \hskip 2mm i\in S_1 \rangle \cap  \langle x_i,y_i \hskip 2mm | \hskip 2mm i\in S_2 \rangle$ only contains elements 
\begin{itemize}
\item[$\cdot$] $x_i$ and $y_i$ for any common vertex $i\in S_1 \cap S_2$.

\item[$\cdot$] $ x_ix_j$, $x_iy_j$, $y_ix_j$ and  $y_iy_j$ for any pair of vertices  $i\in S_1 \setminus S_2$ and $j\in S_2 \setminus S_1$.
\end{itemize}


 $\cdot$ A minimal set of generators of $\langle x_i,y_i \hskip 2mm | \hskip 2mm i\in S_1 \rangle \cap J_{\widetilde{H}_{2,t}}$  only  contains elements 
 \begin{itemize}
\item[$\cdot$]  $\Delta_{i j}$ whenever
 the edge $\{i,j\}$ belongs to the complete graph $\widetilde{H}_{2,t}$ and either $i$ or $j$ is a vertex of $S_1$.

\item[$\cdot$] $x_i \Delta_{k\ell}$ and $y_i \Delta_{k\ell}$ whenever $i \in S_1$ and the edge $\{k,\ell\}$ belongs to $\widetilde{H}_{2,t}$ but $k,\ell \not \in S_1$.
\end{itemize}

\hskip .3cm Analogously for $\langle x_i,y_i \hskip 2mm | \hskip 2mm i\in S_2 \rangle \cap J_{\widetilde{H}_{1,s}}$.


$\cdot$ A minimal set of generators of $J_{\widetilde{H}_{1,s}} \cap J_{\widetilde{H}_{2,t}} $ only contains elements 
 \begin{itemize}
\item[$\cdot$]  $\Delta_{i j}$ for common edges $\{i,j\}$.

\item[$\cdot$]  $\Delta_{i j} \Delta_{k\ell}$ for edges $\{i,j\}$ in $\widetilde{H}_{1,s}$ and $\{k,\ell\}$ in $\widetilde{H}_{2,t}$ sharing at most one vertex.

\item[$\cdot$] $  x_j y_\ell \Delta_{i k} - x_\ell y_k  \Delta_{i j}=  x_i y_ k \Delta_{j \ell} - x_j y_i  \Delta_{k \ell} $ for edges $\{i,j\}$ in $\widetilde{H}_{1,s}$, $\{k,\ell\}$ in $\widetilde{H}_{2,t}$ and $\{j,k\}$ being a common edge.

\end{itemize}


\end{discussion}

Assumption (A1) follows from the following 

  \begin{proposition}
 Let $J_G$ be the binomial edge ideal associated to a graph $G$  on the set of vertices $[n]$. Then, the poset $\cQ_{J_G}$ is a subset of a distributive lattice.
  \end{proposition}
  
  \begin{proof}
  Given  prime ideals $I_{q_1}, I_{q_2}, I_{q_3}$ in the poset $\cQ_{J_G}$ we have to check 
  $$(I_{q_1} + I_{q_2})\cap I_{q_3} = ( I_{q_1} \cap I_{q_3}) +  ( I_{q_2} \cap I_{q_3}).$$ The inclusion  $\supseteq$ always holds so let's prove the other inclusion.  Assume that 
$$I_{q_1}=P_{S_1}(H_1)= \langle x_i,y_i \hskip 2mm | \hskip 2mm i\in S_1 \rangle + J_{\widetilde{H}_{1,1}} +\cdots + J_{\widetilde{H}_{1,c_{q_1}}},$$
$$I_{q_2}=P_{S_2}(H_2)= \langle x_j,y_j \hskip 2mm | \hskip 2mm j\in S_2 \rangle + J_{\widetilde{H}_{2,1}} +\cdots + J_{\widetilde{H}_{2,c_{q_2}}},$$
$$I_{q_3}=P_{S_3}(H_3)= \langle x_j,y_j \hskip 2mm | \hskip 2mm j\in S_3\rangle + J_{\widetilde{H}_{3,1}} +\cdots + J_{\widetilde{H}_{3,c_{q_3}}}. $$
 Recall that $I_{q_1} + I_{q_2}$ is not necessarily a prime ideal. Indeed,  we have a presentation
$$I_{q_1} + I_{q_2}= \langle x_i,y_i \hskip 2mm | \hskip 2mm i\in S_1\cup S_2 \rangle + J_{{K}_1} +\cdots + J_{{K}_{c}},$$ 
where  $K_t$ are non necessarily complete graphs corresponding to  the connected components of $K\setminus S_1 \cup S_2$ with $K$ being a graph in $[n]$. A minimal set of generators of $(I_{q_1} + I_{q_2})\cap I_{q_3}$   is contained in the intersection of the prime ideal $\langle x_i,y_i \hskip 2mm | \hskip 2mm i\in S_1\cup S_2 \rangle + J_{\widetilde{K}_1} +\cdots + J_{\widetilde{K}_{c}}$ with $ I_{q_3}$. In particular these elements have the form considered in Discussion \ref{discussion}. 
  
  \vskip 2mm
  
 Now,  let $u\in (I_{q_1} + I_{q_2})\cap I_{q_3}$ be an element in a  minimal set of generators. 
  
\begin{itemize}  
\item[$\cdot$] If $u=x_i$ or $y_i$,   then $i\in S_1\cup S_2$ and  $i\in S_3$ and thus $u\in ( I_{q_1} \cap I_{q_3}) +  ( I_{q_2} \cap I_{q_3})$.

\vskip 2mm
  
\item[$\cdot$] If $u$ is $x_ix_j$, $x_iy_j$, $y_ix_j$ or $y_iy_j$ then $i\in S_1\cup S_2 \setminus S_3$ and $j\in S_3 \setminus S_1\cup S_2$ and thus  $u\in ( I_{q_1} \cap I_{q_3}) +  ( I_{q_2} \cap I_{q_3})$.

\vskip 2mm
 
\item[$\cdot$] If $u=\Delta_{ij} $ it could be the case that the edge $\{i,j\}$ belongs to ${K}_{t}$, and consequently is an edge of $H_1$
or $H_2$, and either $i$ or $j$ is a vertex of $S_3$. It could also be the way around, $\{i,j\}$ belongs to $H_3$ and  either $i$ or $j$ is a vertex of $S_1 \cup S_2$.
It may also happen that  $\{i,j\}$ is a common edge of ${K}_{t}$ and $H_3$. In any case we have $u\in ( I_{q_1} \cap I_{q_3}) +  ( I_{q_2} \cap I_{q_3})$.

\vskip 2mm
 
\item[$\cdot$] If $u$ is $x_i \Delta_{k\ell}$ or $y_i \Delta_{k\ell}$ then $i \in S_1\cup S_2$ and the edge $\{k,\ell\}$ belongs to $H_3$ but $k,\ell \not \in S_1\cup S_2$ or the way around, $i \in S_3$ and the edge $\{k,\ell\}$ belongs to ${K}_{t}$ but $k,\ell \not \in S_3$.  We have $u\in ( I_{q_1} \cap I_{q_3}) +  ( I_{q_2} \cap I_{q_3})$.

\vskip 2mm
  
\item[$\cdot$] If $u=\Delta_{ij} \Delta_{k\ell}$ then the edges $\{i,j\}$ in $K_{t}$ and $\{k,\ell\}$ in $H_3$ share at most one vertex. Since 
$\{i,j\}$ belongs to $H_1$ or $H_2$, we  have $u\in ( I_{q_1} \cap I_{q_3}) +  ( I_{q_2} \cap I_{q_3})$.

\vskip 2mm

\item[$\cdot$] If $u= x_j y_\ell \Delta_{i k} - x_\ell y_k  \Delta_{i j}=  x_i y_ k \Delta_{j \ell} - x_j y_i  \Delta_{k \ell} $ then $\{i,j\}$ is an edge in $K_t$ and consequently an edge of $H_1$ or $H_2$, $\{k,\ell\}$ in $H_3$ and $\{j,k\}$ is a common edge. We  have $u\in ( I_{q_1} \cap I_{q_3}) +  ( I_{q_2} \cap I_{q_3})$.

\end{itemize} 
  \end{proof}
  
  \begin{remark} \label{cofinal}
The prime ideals in the minimal primary decomposition of the binomial edge ideal $J_G$ are the maximal elements of the
posets $\cQ_{J_G}$ and $\cP_{J_G}$. 
One deduces that the inverse systems $( A/I_q)_{q\in \cQ_{J_G}}$   and $( A/I_p)_{p\in \cP_{J_G}}$  are cofinal. 
In particular
$$\lim_{q\in \cQ_{J_G}} A/I_q =\lim_{p\in \cP_{J_G}} A/I_p.$$
This limit is isomorphic to $A/I$ under the conditions of Assumption (A1).
\end{remark}

Now we are ready to present the main result of this work which is a Hochster's type decomposition for the local cohomology modules associated to a binomial edge ideal.

\begin{theorem}\label{main}
Let $A=\K[x_1,\dots, x_n, y_1,\dots, y_n]$ be a polynomial ring with coefficients over a field $\K$ and $\fM$ be its homogeneous maximal ideal. Let $J_G \subseteq A$ be the binomial edge ideal associated to a graph $G$ in the set of vertices $[n]$. Let $\cQ_{J_G}$ be the poset associated to a minimal primary decomposition of $J_G$. Then, the local cohomology modules with respect to $\fM$ of $A/J_G$ admit the following decomposition as $\K$-vector spaces:
\[
H_{\mathfrak{m}}^{r} \left(A/J_G\right)\cong\bigoplus_{q\in \cQ_{J_G}} H_{\mathfrak{m}}^{d_q} \left(A/I_q\right)^{\oplus M_{r,q}}
\]
where $M_{r,q}=\dim_{\K} \widetilde{H}^{r-d_q-1} ((q,1_{{\cQ_{J_G}}});\K).$ Moreover we have a decomposition as graded $\K$-vector spaces.
\end{theorem}

\begin{example}
Let $A=\K[x_1,\dots, x_8, y_1,\dots, y_8]$ be a polynomial ring with coefficients over a field $\K$ and let
$J_G \subseteq A$ be the binomial edge ideal associated to the complete bipartite graph $G=K_{3,5}$.
Its minimal primary decomposition is $J_G= I_{p_1} \cap I_{p_2} \cap I_{p_3}$ where

$I_{p_1}=\langle x_1,x_2,x_3,x_4,x_5,y_1,y_2,y_3,y_4,y_5\rangle$

$I_{p_2}=\langle \Delta_{12}, \dots , \Delta_{78}\rangle$

$I_{p_3}=\langle x_6,x_7,x_8,y_6,y_7,y_8\rangle$

\vskip 2mm

\noindent The associated poset $\cQ_{J_G}$ has the form 
 $${\tiny
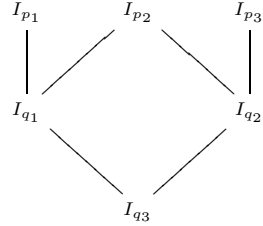
{\xymatrix
{   I_{p_1}  \ar@{-}[d] & I_{p_2} \ar@{-}[dr]\ar@{-}[dl]& I_{p_3}  \ar@{-}[d] \\ I_{q_1} \ar@{-}[dr]& & I_{q_2} \ar@{-}[dl] \\ &I_{q_3}& }}}$$
where 

\vskip 2mm

$I_{q_1}=  \langle x_1,x_2,x_3,x_4,x_5,y_1,y_2,y_3,y_4,y_5, \Delta_{67}, \Delta_{68}, \Delta_{78} \rangle$

$I_{q_2}=  \langle x_6,x_7,x_8,y_6,y_7,y_8, \Delta_{12},  \Delta_{13}, \Delta_{14}, \Delta_{15},  \Delta_{23}, \Delta_{24}, \Delta_{25},  \Delta_{34}, \Delta_{35},  \Delta_{45} \rangle$

$I_{q_3}=  \langle x_1,x_2,x_3,x_4,x_5,x_6,x_7,x_8, y_1,y_2,y_3,y_4,y_5,y_6,y_7,y_8 \rangle$

\vskip 2mm

\noindent Then:
$$\begin{tabular}{|c|c|c|c|c|c|}\hline
$q$ &  $(q,1_{{\cQ_{J_G}}})$ & $\dim _k \tilde{H}^{-1} $ & $\dim _k
\tilde{H}^{0} $
 & $\dim _k \tilde{H}^{1} $\\
 \hline  $p_1$  & $\emptyset$ & 1 & - & -\\
  $p_2$ &  $\emptyset$ & $1$ & - & -\\
  $p_3$ &  $\emptyset$ & $1$ & - & -\\
  $q_1$ &   $\bullet \bullet $ & - & $1$& -  \\ 
  $q_2$ &   $\bullet \bullet $ & - & $1$& -  \\
  $q_3$ &   $ { \mid}{\hskip -.5mm \diagup\diagdown} { \hskip -.6mm \mid}$ & - &  - & -  \\ \hline
\end{tabular}
$$

\vskip 2mm

\noindent It follows from Theorem \ref{main} that 

\vskip 2mm

$H_\fM^{5}(A/J_G) = H_\fM^{4}(A/I_{q_1}),$ \hskip 2mm  $H_\fM^6(A/J_G) = H_\fM^6(A/I_{p_1})$, \hskip 2mm $H_\fM^7(A/J_G) = H_\fM^{6}(A/I_{q_2})$,

\vskip 2mm

$H_\fM^9(A/J_G) = H_\fM^9(A/I_{p_2})$, \hskip 2mm  $H_\fM^{10}(A/J_G) = H_\fM^{10}(A/I_{p_3})$.

\vskip 2mm

\noindent In particular, there are five local cohomology modules different from zero so $A/J_G$ is not Cohen-Macaulay.
 
\end{example}

From the vanishing of local cohomology modules we may deduce a simple criterion for the Cohen-Macaulayness of a binomial edge ideal in terms of the associated poset.

\begin{corollary} \label{CM}
Let $A=\K[x_1,\dots, x_n, y_1,\dots, y_n]$ be a polynomial ring with coefficients over a field $\K$ and $\fM$ be its homogeneous maximal ideal. 
Let $J_G \subseteq A$ be the binomial edge ideal associated to a graph $G$ in the set of vertices $[n]$. Let $\cQ_{J_G}$ be the poset associated to a 
minimal primary decomposition of $J_G$. Then the following are equivalent:
\begin{itemize}
 \item[i)] $J_G$ is Cohen-Macaulay.
 \item[ii)] $M_{r,q}=\dim_{\K} \widetilde{H}^{r-d_q-1} ((q,1_{{\cQ_{J_G}}});\K)=0$ for all $r\neq \dim A/J_G$ and all $q\in \cQ_{J_G}$.
\end{itemize}
\end{corollary}

\vskip 2mm

Analogously we can also provide a criterion for the Buchsbaumness of a binomial edge ideal.

\vskip 2mm

\begin{corollary} \label{Buch}
Let $A=\K[x_1,\dots, x_n, y_1,\dots, y_n]$ be a polynomial ring with coefficients over a field $\K$ and $\fM$ be its homogeneous maximal ideal. 
Let $J_G \subseteq A$ be the binomial edge ideal associated to a graph $G$ in the set of vertices $[n]$. Let $\cQ_{J_G}$ be the poset associated to a 
minimal primary decomposition of $J_G$. Then the following are equivalent:
\begin{itemize}
 \item[i)] $J_G$ is Buchsbaum.
 \item[ii)] $M_{r,q}=\dim_{\K} \widetilde{H}^{r-d_q-1} ((q,1_{{\cQ_{J_G}}});\K)=0$ for all $r\neq \dim A/J_G$ and all $q\in \cQ_{J_G}$ such that $I_q \neq \fM$.
\end{itemize}
\end{corollary}

\vskip 2mm

Using Theorem \ref{main} and some results obtained by Conca and Herzog in \cite{CH} we are able to provide a formula for the $\Z$-graded Hilbert series.


\vskip 2mm

Let $HS(M;t):= \sum_{a\in \Z} \dim_{\K} (M_a) \hskip 1mm t^a $ be the $\Z$-graded Hilbert series of a $\Z$-graded module $M$ such that 
$\dim_{\K} (M_a) < +\infty$ for all $a\in \Z$. In the case that $M$ is the local cohomology of a binomial edge ideal we have:

\begin{theorem}\label{Hilbert}
Let $A=\K[x_1,\dots, x_n, y_1,\dots, y_n]$ be a polynomial ring with coefficients over a field $\K$ and $\fM$ be its homogeneous maximal ideal. Let $J_G \subseteq A$ be the binomial edge ideal associated to a graph $G$ in the set of vertices $[n]$. Let $\cQ_{J_G}$ be the poset associated to a minimal primary decomposition of $J_G$. Then
\begin{align*}
 HS\left(H_{\mathfrak{m}}^r \left(A/J_G\right);t\right) & =   \sum_{\substack{ q\in \cQ_{J_G} \\ I_q \hskip 2mm {\it monomial}}}  M_{r,q}  \hskip 2mm  \frac{t^{-d_q}}{(1- t^{-1})^{d_q}}\\
& + \sum_{\substack{ q\in \cQ_{J_G} \\ I_q \hskip 2mm {\it not \hskip 1mm monomial}}}   M_{r,q}  \hskip 2mm  
\frac{\prod_{i=1}^{c_q} ((n_{i,q}-1) t^{-n_{i,q}} + t^{-(n_{i,q}+1)})}{(1-t^{-1})^{d_q}} 
\end{align*}
where 
$M_{r,q}=\dim_{\K} \widetilde{H}^{r-d_q-1} ((q,1_{\cQ_{J_G}});
\K).$ 
\end{theorem}

\begin{proof}
From the decomposition given in Theorem \ref{main}
and the additivity of the Hilbert series with respect to short exact sequences we have
\[
HS\left(H_{\mathfrak{m}}^r \left(A/J_G\right);t\right)=\sum_{q\in \cQ_{J_G}}  M_{r,q}  \hskip 1mm HS\left(H_{\mathfrak{m}}^{d_q}
\left(A/I_q\right);t\right),
\]

If $I_q=\langle x_i, y_i \hskip 2mm | \hskip 2mm i\in S \rangle$ is a monomial ideal, then, using Hochster's formula we have $$HS\left(H_{\mathfrak{m}}^{d_q}
\left(A/I_q\right);t\right)= \left( \frac{t^{-1}}{1- t^{-1}}\right)^{d_q}.$$

If $I_q$ is not a monomial ideal, we assume  that
$A/I_q =  A_1/J_{\widetilde{H}_1}  \otimes_\K \cdots \otimes_\K A_c/J_{\widetilde{H}_{c_q}} .$
Given the relation given by the Hilbert series of a Cohen-Macaulay ring and its canonical module we get
\begin{align*}
HS\left(\omega_{A/I_q} ;t\right) &=(-1)^{d_q} HS\left(A/I_q ;t^{-1}\right) = (-1)^{d_q} \prod_{i=1}^{c_q} HS\left(A_i/J_{\widetilde{H}_i}; t^{-1}\right)\\
&= (-1)^{d_q} \prod_{i=1}^{c_q} (-1)^{d_i} HS\left(\omega_{A_i/J_{\widetilde{H}_i}};t\right) =\prod_{i=1}^{c_q} HS\left(\omega_{A_i/ J_{\widetilde{H}_i}};t\right)\\
&= \prod_{i=1}^{c_q} \frac{t^{n_{i,q}}((n_{i,q}-1)  + t)}{(1-t)^{d_i}}
\end{align*}
where the last assertion follows from \cite{CH}. Then, taking into account that graded local duality reverses the degrees, we get
$$HS\left(H_{\mathfrak{m}}^{d_q} 
\left(A/I_q \right);t\right)= \prod_{i=1}^{c_q} HS\left(H_{\mathfrak{m_i}}^{d_i}\left(A_i/J_{\widetilde{H}_i}\right);t\right)
= \prod_{i=1}^{c_q} \frac{ (t^{-1})^{n_{i,q}}((n_{i,q}-1) + t^{-1})}{(1-t^{-1})^{d_i}} $$
and the result follows.
\end{proof}

\begin{example}
Let $A=\K[x_1,\dots, x_8, y_1,\dots, y_8]$ be a polynomial ring with coefficients over a field $\K$ and let
$J_G \subseteq A$ be the binomial edge ideal associated to the complete bipartite graph $G=K_{3,5}$. Then

\vskip 2mm

$HS\left( H_\fM^{5}(A/J_G) ;t\right) = HS\left(H_\fM^{4}(A/I_{q_1});t\right)=\frac{ 2t^{-3}+t^{-4}}{(1-t^{-1})^4},$ 

\hskip 2mm  

$HS\left(H_\fM^6(A/J_G) ;t\right) = HS\left(H_\fM^6(A/I_{p_1});t\right)= \frac{t^{-6}}{( 1- t^{-1})^6} $, 

\hskip 2mm 

$HS\left(H_\fM^7(A/J_G) ;t\right) = HS\left(H_\fM^{6}(A/I_{q_2});t\right)= \frac{4t^{-5}+t^{-6}}{(1-t^{-1})^6} $,

\vskip 2mm  

$HS\left(H_\fM^9(A/J_G) ;t\right) = HS\left(H_\fM^9(A/I_{p_2});t\right)=\frac{7t^{-8}+t^{-9}}{(1-t^{-1})^9}$, 

\hskip 2mm  

$HS\left(H_\fM^{10}(A/J_G) ;t\right) = HS\left(H_\fM^{10}(A/I_{p_3});t\right)=\frac{t^{-10}}{ (1- t^{-1})^{10}}$.

\vskip 2mm

\end{example}

\section{Local cohomology modules of generic initial ideals of binomial edge ideals} \label{sec-gin}

Let $J_G \subseteq \K[x_1,\dots,x_n, y_1,\dots,y_n]$ be the binomial edge ideal associated to a graph $G$ on $[n]$. 
A precise description of the corresponding generic initial ideal $\gin(J_G)$ has been given by Conca, De Negri and Gorla  in \cite{CDG4}
in order to prove that binomial edge ideals belong to the class of {\it Cartwright-Sturmfels ideals} introduced in \cite{CS}. Namely, we have:

\begin{theorem} \cite[Theorem 2.1]{CDG4}\label{mainbinedge}
Let $J_G$ be the binomial edge ideal associated to a graph $G$ on $[n]$ .
Then, the $\bZ^n$-graded generic initial ideal of $J_G$ is 
$$\gin(J_G):= \langle  x_ix_j y_{a_1}\cdots y_{a_v}  \hskip 2mm | \hskip 2mm  \{i, a_1,  \cdots,  a_v, j\}  \hskip 2mm 
\mbox{is a path in} \hskip 2mm G \rangle.$$
\end{theorem} 

An important feature of generic initial ideals is that they behave well with respect to minimal primary decompositions.

  \begin{proposition}  \cite[Corollary 1.12]{CDG4} \label{IntersectionPrimesgin}
Let $J_G=P_{S_1}(G) \cap \cdots \cap P_{S_r}(G)$ be the minimal primary
decomposition of the binomial edge ideal $J_G$ associated to a graph $G$ on $[n]$ . Then we have
 a decomposition $$\gin(J_G)=\gin(P_{S_1}(G)) \cap \cdots \cap \gin(P_{S_r}(G)).$$

\end{proposition}
  
  If we take a close look, we will see that the decomposition of $\gin(J_G)$ that we obtain is not necessarily a minimal primary decomposition.
  
  \begin{proposition}\label{gin}
  Let $G$ be a graph on $[n]$, and $S\subseteq [n]$. Then:
  
  \begin{itemize}
\item[i)] $\gin(P_S(G)) = \langle x_i,y_i \hskip 2mm | \hskip 2mm i\in S \rangle + \gin(J_{\widetilde{G}_1}) +\cdots +\gin(J_{\widetilde{G}_{c}})$.  In particular, it is a squarefree monomial ideal  of height $n-c+|S|$ for every $S\subseteq [n]$.

\vskip 2mm

\item[ii)] If $\widetilde{G}$ is a complete graph on the vertices $\{1, \dots , t\} \subseteq [n]$, then 
$$\gin(J_{\widetilde{G}})=\langle x_i x_j \hskip 2mm | \hskip 2mm 1\leq i<j\leq t  \rangle =\bigcap_{i=1}^t \langle x_1, \dots \widehat{x_i}, \dots , x_t \rangle$$ which is a Cohen-Macaulay monomial ideal.

\end{itemize}

  \end{proposition}

\subsection{A poset associated to the generic initial ideal of a binomial edge ideal} \label{poset2}
  Let $G$ be a graph on the set of vertices $[n]$ and let $J_G$ be its associated binomial edge ideal. Its generic ideal $\gin(J_G)$ is monomial so we may use the original Hochster's formula to describe its local cohomology modules.  However, since our goal is to compare the local cohomology modules of both ideals, we are going to use Theorem \ref{sucesion espectral}  considering a poset associated to $\gin(J_G)$ that will be more useful for our purposes.
  
  \vskip 2mm
  
Let $J_G=P_{S_1}(G) \cap \cdots \cap P_{S_r}(G)$ be the minimal primary
decomposition of the binomial edge ideal $J_G$.   By Proposition \ref{IntersectionPrimesgin}
we have the decomposition $$\gin(J_G)=\gin(P_{S_1}(G)) \cap \cdots \cap \gin(P_{S_r}(G)).$$ 
Associated to this decomposition we are going to consider the poset $\cQ_{\gin(J_G)}$ which contains $\gin(I_q)$ for all ideals $I_q$ in  
$\cQ_{J_G}$. It is clear that $\cQ_{\gin(J_G)}$ is the same poset as $\cQ_{J_G}$ where we simply have a different label for the ideals so, if no confusion arise, we will simple denote both as $\cQ_{J_G}$.
In this way we  emphasize that the decomposition obtained in Theorem \ref{main} and Theorem \ref{main2}  for the local cohomology of $J_G$ and 
$\gin(J_G)$ respectively, are the same since the order complexes $(q,1_{{\cQ_{J_G}}})$ are exactly the same in both cases and thus they have the same reduced cohomology. 

\vskip 2mm


\subsection{Local cohomology of the generic initial ideal of a binomial edge ideal}

In this Section we present a decomposition of the local cohomology modules of the generic initial ideal of a binomial edge ideal which is coarser than the original Hochster's formula for monomial ideals but it will be more appropriate for our purposes. In order to apply Theorem \ref{Hochster}
we have to check first that the poset $\cQ_{J_G}$ associated to the generic initial ideal of a binomial edge ideal satisfies Assumptions (A1), (A2) and (A3). 

\vskip 2mm 

In this context we are dealing with  squarefree monomial ideals so $\cQ_{J_G}$ is a subset of a distributive lattice of ideals of $A$
so Assumption {\rm (A1)} is satisfied. Indeed 
$$\lim_{q\in \cQ_{J_G}} A/\gin(I_q)  \cong A/{\gin(J_G)}$$ since $\cQ_{J_G}$ is cofinal with the poset $\cP_{\gin(J_G)}$ associated to a minimal primary decomposition 
of $\gin(J_G)$. We have
$$A/\gin(I_q) =  A_1/\gin(J_{\widetilde{H}_1})  \otimes_\K \cdots \otimes_\K A_c/\gin(J_{\widetilde{H}_c}) ,$$
so it  is Cohen-Macaulay by Proposition \ref{gin} and Theorem \ref{CMtensor} so Assumption {\rm (A2)} is also satisfied.

\vskip 2mm 

In order to prove Assumption {\rm (A3)} we take $p\neq q$ with $\hlt (\gin(I_q)) > \hlt (\gin(I_p))$ 
and we want to check that  $\gin(I_q)$ is not contained in any minimal prime of $\gin(I_p)$. 
Recall that these ideals are of the form

\hskip 5mm $ \gin(I_p)= \gin(P_{S_p}(H_p))= \langle x_i,y_i \hskip 2mm | \hskip 2mm i\in S_p \rangle + \gin(J_{\widetilde{H}_{p,1}}) +\cdots +\gin(J_{\widetilde{H}_{p,d}})$

\hskip 5mm $  \gin(I_q)= \gin(P_{S_q}(H_q))=\langle x_j,y_j \hskip 2mm | \hskip 2mm j\in S_q \rangle + \gin(J_{\widetilde{H}_{q,1}}) +\cdots +\gin(J_{\widetilde{H}_{q,c}})$

\noindent for some  graphs $H_p$ and $H_q$ in the vertex set $[n]$ and some $S_p, S_q\subseteq [n]$. Indeed, these graphs may not contain all the vertices in $[n]$. Notice first that Assumption {\rm (A3)} is satisfied when $S_p = \emptyset$ just because there is $y_j \in \gin(I_q)$ that does not belong to any minimal prime of  $ \gin(I_p)$.  If $S_p \neq \emptyset$, 
we may assume without loss of generality that we are using the variables
$$\{\underbrace{x_{1},\dots , x_{a}, y_{1},\dots , y_{a}}_{S_p}, \underbrace{x_{a + 1},  \dots , x_{b}}_{\widetilde{H}_{p,1}}, \dots ,  \underbrace{x_{\ell +1},  \dots , x_{m}}_{\widetilde{H}_{p,d}} \}$$
with $1 \leq  a < b < \cdots < \ell < m \leq n$,  accordingly to the description of $ \gin(I_p)$. Therefore we have the following minimal primary decomposition 
\begin{align*}
 \gin(I_p) &= \langle x_{1},\dots , x_{a}, y_{1},\dots , y_{a} \rangle  + \\ & + \bigcap_{ j_1 = a+1} ^{b} \langle x_{a+1}, \dots , \widehat{x_{j_1}}, \dots , x_{b} \rangle + \cdots + \bigcap_{  j_d= \ell +1}^{m} \langle x_{\ell +1 }, \dots , \widehat{x_{j_d}}, \dots , x_{m} \rangle \\
 & = \bigcap_{ j_1 = a+1} ^{b}\cdots \bigcap_{  j_d= \ell +1}^{m} \langle  x_{1},\dots , x_{a}, y_{1},\dots , y_{a} , x_{a+1}, \dots , \widehat{x_{j_1}}, \dots , x_{b}, \dots , x_{\ell +1}, \dots , \widehat{x_{j_d}}, \dots , x_{m} \rangle.
\end{align*}

Now we are ready to prove Assumption {\rm (A3)}. If $S_q \not \subset S_p$, then there exists $y_{s}$, for some $s \geq a$ such that $y_s \not \in \gin(I_p)$ and, in particular, this element does not belong to any minimal prime of $\gin(I_p)$.

\vskip 2mm

If $S_q \subseteq S_p$ and $\hlt (\gin(I_q)) > \hlt (\gin(I_p))$ then there is at least an element,  say $x_sx_t \in  \gin(I_q)$, 
with both vertices not belonging to the graph $H_p$, that is $s,t \geq m$. Then it is clear that $\gin(I_q)$ is not contained in any minimal prime of $\gin(I_p)$.

\vskip 2mm

Once we know that the poset $\cQ_{J_G}$ satisfies Assumptions (A1), (A2) and (A3), we obtain the following  coarser version of Hochster's formula.

\begin{theorem}\label{main2}
Let $A=\K[x_1,\dots, x_n, y_1,\dots, y_n]$ be a polynomial ring with coefficients over a field $\K$ and $\fM$ be its homogeneous maximal ideal. Let $J_G \subseteq A$ be the binomial edge ideal associated to a graph $G$ in the set of vertices $[n]$. Let $\cQ_{J_G}$ be the poset associated to a minimal primary decomposition of $J_G$. Then the local cohomology modules with respect to $\fM$ of 
$A/\gin(J_G)$ admit the following decomposition as $\K$-vector spaces:

\[
H_{\mathfrak{m}}^{r} \left(A/\gin(J_G)\right)\cong\bigoplus_{p\in \cQ_{J_G}} H_{\mathfrak{m}}^{d_p} \left(A/\gin(I_p)\right)^{\oplus M_{r,p}}
\]
where $M_{i,p}=\dim_{\K} \widetilde{H}^{i-d_p-1} ((p,1_{{\cQ_{J_G}}});
\K).$ Moreover this is a decomposition of graded $\K$-vector spaces.
\end{theorem}

As a consequence of the decomposition results for local cohomology modules obtained in Theorem \ref{main} and Theorem \ref{main2},  we can give an affirmative answer to Conjecture \ref{conjH} for the case of binomial edge ideals.

\begin{theorem}\label{compare}
Let $A=\K[x_1,\dots, x_n, y_1,\dots, y_n]$ be a polynomial ring with coefficients over a field $\K$ and $\fM$ be its homogeneous maximal ideal. Let $J_G \subseteq A$ be the binomial edge ideal associated to a graph $G$ in the set of vertices $[n]$.  Then we have
$$\dim_\K H^r_{\fM}(A/J_G)_a=\dim_\K H^r_{\fM}(A/\gin(J_G))_a$$
for every $r\in \bN$ and every $a\in \Z^n$.
\end{theorem}

\begin{proof}
By Theorem \ref{main} and Theorem \ref{main2}  we only have to check that
$$\dim_\K H^{d_q}_{\fM}(A/I_q)_a=\dim_\K H^{d_q}_{\fM}(A/\gin(I_q))_a$$
for  every $a\in \Z^n$ and every $q\in \cQ_{J_G}$.  
Assuming that the ideal $I_q$ has the form 
$$I_q=P_S(H)= \langle x_i,y_i \hskip 2mm | \hskip 2mm i\in S \rangle + J_{\widetilde{H}_1} +\cdots + J_{\widetilde{H}_{c}}$$  it is then enough to prove that $$\dim_\K H^{d_p}_{\fM_i}(A_i/J_{\widetilde{H_i}})_a=\dim_\K H^{d_p}_{\fM_i}(A_i/\gin(J_{\widetilde{H_i}})))_a$$
for  every $a\in \Z^n$ where ${\widetilde{H_i}}$ is a complete graph.  This equality holds true since $A_i/J_{\widetilde{H_i}}$ and $A_i/\gin(J_{\widetilde{H_i}})$ have the same Hilbert and  they are Cohen-Macaulay so they only have one local cohomology module different from zero.
\end{proof}

%

%

 As a consequence of this result we have that the extremal Betti numbers, the projective dimension and the Castelnuovo-Mumford regularity 
 of $J_G$  and $\gin(J_G)$ are equal. In particular we have:

\begin{proposition}\label{regularity}
Let $A=\K[x_1,\dots, x_n, y_1,\dots, y_n]$ be a polynomial ring with coefficients over a field $\K$ and $\fM$ be its homogeneous maximal ideal. Let $J_G \subseteq A$ be the binomial edge ideal associated to a graph $G$ in the set of vertices $[n]$. Let $\cQ_{J_G}$ be the poset associated to a minimal primary decomposition of $J_G$. Then, 
\[
\reg \left(A/J_G\right)=\max_{r\geq 0,\ q\in \cQ_{J_G}} \{r-d_q \hskip 2mm | \hskip 2mm \ M_{r,p}\neq 0\},
\]
where 
$M_{r,q}=\dim_{\K} \widetilde{H}^{r-d_q-1} ((q,1_{\cQ_{J_G}});
\K).$
\end{proposition}

\begin{proof}
We have
\[
\reg \left(A/J_G\right)=\max_{r\geq 0} \{{\rm end}\left(H_{\mathfrak{m}}^r \left(A/J_G\right)\right)+r\}=\max_{r\geq 0,\ q\in \cQ_{J_G}} \{{\rm end}\left( H_{\mathfrak{m}}^{d_q}
\left(A/I_q\right)\right)+r \hskip 2mm | \hskip 2mm \ M_{r,p}\neq 0\}.
\]
Recall that $\gin(J_{\widetilde{H}_i}) = \langle x_i x_j \hskip 2mm | \hskip 2mm 1\leq i<j\leq n_{i,q}  \rangle$  and thus, from a direct computation, we get ${\rm end}\left(H_{\mathfrak{m_i}}^{d_i} \left(A_i/\gin(J_{\widetilde{H}_i})\right)\right)=-d_i$. Therefore
 $${\rm end}\left(H_{\mathfrak{m}}^{d_q} \left(A/I_q\right)\right)= {\rm end}\left(H_{\mathfrak{m}}^{d_q} \left(A/\gin(I_q)\right)\right)=-d_q$$  and the result follows.
\end{proof}

Finally we turn our attention to the $\Z^n$-graded Hilbert series 
$$HS(M;t_1,\dots , t_n):= \sum_{a\in \Z^n} \dim_{\K} (M_a) \hskip 1mm t_1^{a_1} \cdots t_n^{a_n}$$ 
where $a=(a_1,\dots, a_n) \in \Z^n$. The formula that we obtain is a little bit involved. To start with, for each ideal $I_q \in \cQ_{J_G} $, we will rename and reorder the variables $t_i$ accordingly so we have 
$$\{t_1,\dots, t_n\} = \{ t^0_1,\dots, t^0_{|S|}, t^1_1,\dots, t^1_{n_{1,q}}, \dots , t^{c_q}_1,\dots, t^{c_q}_{n_{c_q,q}}\}.$$
In the case that $I_q=\langle x_i, y_i \hskip 2mm | \hskip 2mm i\in S \rangle$ is a monomial ideal, we rename and reorder the variables as
$$\{t_1,\dots, t_n\} = \{ t_{1,q},\dots, t_{|S|,q}, t_{|S|+1,q},\dots, t_{n,q}\}.$$

\begin{theorem}\label{Hilbert2}
Let $A=\K[x_1,\dots, x_n, y_1,\dots, y_n]$ be a polynomial ring with coefficients over a field $\K$ and $\fM$ be its homogeneous maximal ideal. Let $J_G \subseteq A$ be the binomial edge ideal associated to a graph $G$ in the set of vertices $[n]$. Let $\cQ_{J_G}$ be the poset associated to a minimal primary decomposition of $J_G$. Then
\begin{align*}
 HS\left(H_{\mathfrak{m}}^r \left(A/J_G\right);t_1,\dots , t_n\right) & =  \sum_{\substack{ q\in \cQ_{J_G} \\ I_q \hskip 2mm {\it monomial}}}  M_{r,q} \prod_{i=|S|+1}^{n} \frac{(t_{i,q})^{-2}}{(1- (t_{i,q})^{-1})^2} \\
& \hskip -3.5cm + \sum_{\substack{ q\in \cQ_{J_G} \\ I_q \hskip 2mm {\it not \hskip 1mm monomial}}}   M_{r,q}  \hskip 1mm  
\prod_{i=1}^{c_q} \frac{(t_1^i)^{-1} \cdots (t_{n_{i,q}}^i)^{-1} }{(1- (t_1^i)^{-1})\cdots (1-(t_{n_{i,q}}^i)^{-1})} \left( (n_{i,q}-1)+\sum_{j=1}^{n_{i,q}} \frac{ (t_j^i)^{-1}}{(1-(t_j^i)^{-1})}\right)
\end{align*}
where 
$M_{r,q}=\dim_{\K} \widetilde{H}^{r-d_q-1} ((q,1_{\cQ_{J_G}});
\K).$
\end{theorem}

\begin{proof}
From the decomposition given in Theorem \ref{main}
and the additivity of Hilbert series on short exact sequences we have

\begin{align*}
HS\left(H_{\mathfrak{m}}^r \left(A/J_G\right);t_1,\dots, t_n\right) &= \sum_{q\in \cQ_{J_G}}  M_{r,q}  \hskip 1mm HS\left(H_{\mathfrak{m}}^{d_q}
\left(A/I_q\right);t_1,\dots, t_n\right)
\end{align*}
In the case that $I_q=\langle x_i, y_i \hskip 2mm | \hskip 2mm i\in S \rangle$ is  a monomial ideal we have 
\begin{align*}
HS\left(H_{\mathfrak{m}}^{d_q}
\left(A/I_q\right);t_1,\dots, t_n\right) &= \prod_{i=|S|+1}^{n} \frac{(t_{i,q})^{-2}}{(1- (t_{i,q})^{-1})^2} 
\end{align*}
If $I_q$ is not a monomial ideal then 
\begin{align*}
HS\left(H_{\mathfrak{m}}^{d_q}
\left(A/I_q\right);t_1,\dots, t_n\right) &=   
\prod_{i=1}^{c_q} HS\left(H_{\fM_i}^{d_i}(A_i/J_{\widetilde{H}_i});t_1^i,\dots, t_{n_{i,q}}^i\right)\\
&=    \hskip 1mm  
\prod_{i=1}^{c_q} HS\left(H_{\fM_i}^{d_i}(A_i/\gin(J_{\widetilde{H}_i}));t_1^i,\dots, t_{n_{i,q}}^i\right)\\
& \hskip -2cm = \prod_{i=1}^{c_q} \frac{(t_1^i)^{-1} \cdots (t_{n_{i,q}}^i)^{-1} }{(1- (t_1^i)^{-1})\cdots (1-(t_{n_{i,q}}^i)^{-1})} \left( (n_{i,q}-1)+\sum_{j=1}^{n_{i,q}} \frac{ (t_j^i)^{-1}}{(1-(t_j^i)^{-1})}\right)
\end{align*}

The last assertion follows from a direct computation using the usual Hochster's formula for the $\bZ^n$-graded Hilbert series.
\end{proof}


\end{document}